\documentclass[12pt]{amsart}
\usepackage[latin1]{inputenc}
\usepackage{color}

\usepackage{pdfpages}
\usepackage{listings}
\usepackage{float}
\floatstyle{ruled} 
\usepackage{caption}
\usepackage{amsmath,amssymb}
\usepackage{bbold}
\usepackage{enumerate}
\newtheorem{theorem}{Theorem}[section]

\newtheorem{lemma}[theorem]{Lemma}

\theoremstyle{definition}

\newtheorem{remark}[theorem]{Remark}

\numberwithin{equation}{section}

\makeatletter
\@ifpackageloaded{showkeys}{\oddsidemargin=3cm\evensidemargin=3cm}{\oddsidemargin=0.7cm\evensidemargin=0.7cm}\makeatother
\newcommand{\bZ}{\mathbb{Z}}

\newcommand{\supp}{\operatorname{supp}}

\newcommand{\dif}{\,\mathrm{d}}

\newcommand{\tP}{\widetilde{P}}
\newcommand{\tN}{\widetilde{N}}

\newcommand{\td}{\widetilde{d}}
\newcommand{\charfun}{\mathbb 1}
\newcommand{\indset}{\mathfrak i}

\allowdisplaybreaks

\begin{document}

% Preamble
\title{Orthogonal projectors onto spaces of periodic splines}

\author[M. Passenbrunner]{Markus Passenbrunner}
\address{Institute of Analysis, Johannes Kepler University Linz, Austria, 4040 Linz, Altenberger Strasse 69}
\email{markus.passenbrunner@jku.at}
\thanks{The author is supported by the FWF, project number P27723}

\subjclass[2010]{40A05, 41A15, 46E30}
\keywords{Periodic splines, Almost everywhere convergence}
\date{\today}
\begin{abstract}
	The main result of this paper is a proof that for any integrable function $f$ on the torus, 
	any sequence of its orthogonal projections $(\tP_n f)$ onto periodic spline spaces with arbitrary knots
	$\widetilde \Delta_n$ and arbitrary polynomial degree converges  to $f$
	almost everywhere with respect to the Lebesgue measure, provided the
	mesh diameter $|\widetilde\Delta_n|$ tends to zero. We also give a proof of the fact that the operators
	$\tP_n$ are bounded on $L^\infty$ independently of the knots $\widetilde\Delta_n$.
\end{abstract}
\maketitle
\section{Introduction}
\subsection{Splines on an interval}
In this article we prove some results about the periodic spline orthoprojector. In order to achieve this, we rely
on existing results for the non-periodic
spline orthoprojector on a compact interval, so
we first describe some of those results for the latter operator.
Let $k\in \mathbb N$ and $\Delta=(t_i)_{i=\ell}^{r+k}$ 
a knot sequence satisfying
\begin{align*}
	t_i &\leq t_{i+1}, \qquad t_i<t_{i+k}, \\
	t_{\ell}&= \cdots = t_{\ell+k-1} , \qquad t_{r+1} = \cdots = t_{r+k}.
\end{align*}
 Associated to this knot sequence, we define $(N_i)_{i=\ell}^r$ as the sequence of $L^\infty$-normalized B-spline
 functions of order $k$ on  
 $\Delta$ that have the properties
\begin{equation*}
	\supp N_i = [t_i,t_{i+k}],\qquad N_i\geq 0,\qquad \sum_{i=\ell}^r N_i\equiv 1.
\end{equation*}
We write $|\Delta|=\max_{\ell\leq j\leq r}(t_{j+1} - t_j)$ for the maximal mesh width of the partition~$\Delta$.
Then, define the space $\mathcal S_k(\Delta)$ as the set of polynomial splines of order $k$ (or at most degree $k-1$)
with knots $\Delta$, which is the linear span of
the B-spline functions $(N_i)_{i=\ell}^r$. Moreover, let $P_\Delta$ be the orthogonal projection operator onto the
space $\mathcal S_k(\Delta)$ with respect to the ordinary (real) inner product $\langle f,g\rangle =
\int_{t_\ell}^{t_{r+1}}
f(x)g(x)\dif x$, i.e.,
\begin{equation*}
	\langle P_\Delta f,s\rangle = \langle f,s\rangle\qquad \text{for all }s\in\mathcal S_k(\Delta).
\end{equation*}
The operator $P_\Delta$ is also given by the formula
\begin{equation}\label{eq:projBsplines}
	P_\Delta f = \sum_{i=\ell}^r \langle f,N_i\rangle N_i^*,
\end{equation}
where $(N_i^*)_{i=\ell}^r$ denotes the dual basis to $(N_i)$ defined by the relations $\langle N_i^*, N_j\rangle =0$
when $j\neq i$ and $\langle N_i^*, N_i \rangle=1$ for all $i=\ell,\ldots,r$.
%Recall that $P_\Delta$ is characterized by the property
%\begin{equation*}
%	\langle P_\Delta f - f, N_i\rangle=0,\qquad f\in L^2[t_\ell,t_{r+1}], \ \ell\leq i\leq r.
%\end{equation*}
A famous theorem by A. Shadrin states that the $L^\infty$ norm of this projection operator is bounded independently
of the knot sequence~$\Delta$:
\begin{theorem}[\cite{Shadrin2001}]\label{thm:shadrin}
	There exists a constant $c_k$ depending only on the spline order $k$ such that for all
knot sequences $\Delta=(t_i)_{i=\ell}^{r+k}$ as above,
\begin{equation*}
	\|P_\Delta : L^\infty[t_\ell,t_{r+1}]\to L^\infty[t_\ell,t_{r+1}]\| \leq c_k.
\end{equation*}
\end{theorem}
We are also interested in the following equivalent formulation of this theorem, which is proved in \cite{Ciesielski2000}: for a knot sequence
$\Delta$, let $(a_{ij})$ be the matrix $(\langle N_i^*,N_j^*\rangle)$, which is the inverse of the banded
matrix $(\langle N_i,N_j\rangle)$. 
Then, the assertion of Theorem \ref{thm:shadrin}
is equivalent to the existence of two constants $K_0>0$ and $\gamma_0\in (0,1)$ only depending on the spline order
$k$ such that 
\begin{equation}\label{eq:decayineq}
	|a_{ij}| \leq \frac{K_0 \gamma_0^{|i-j|}}{\max\{\kappa_i,\kappa_j\}}, \qquad \ell\leq i,j\leq r,
\end{equation}
 where $\kappa_i$ denotes the length of $\supp N_i$.
The proof of this equivalence uses Demko's theorem
\cite{Demko1977} on the geometric decay of inverses of band
matrices and de Boor's stability (see \cite{deBoor1973} or \cite[Chapter 5, Theorem 4.2]{DeVoreLorentz1993}) which
states that for $0<p\leq \infty$, the $L^p$ norm of a B-spline series is equivalent to a weighted $\ell^p$ norm of its coefficients,
i.e. there exists a constant $D_k$ only depending on the spline order $k$ such that:
\begin{equation*}
	D_k k^{-1/p}\Big(\sum_{j} |c_j|^p \kappa_j \Big)^{1/p} \leq \Big\| \sum_{j} c_j N_j\Big\|_{L^p} \leq
	\Big(\sum_{j} |c_j|^p \kappa_j \Big)^{1/p}.
\end{equation*}

In fact, for $a_{ij}$, we actually have the following improvement of \eqref{eq:decayineq}
(see \cite{ShadrinPassenbrunner2014}): There exist two constants $K>0$ and
$\gamma\in (0,1)$ that depend only on the spline order $k$ such that
\begin{equation}
	\label{eq:improvedversion}
	|a_{ij}| \leq \frac{K\gamma^{|i-j|}}{h_{ij}}, \qquad \ell\leq i,j\leq r,
\end{equation}
where $h_{ij}$ denotes the length of the convex hull of $\supp N_i\cup\supp N_j$. 
This inequality can be used to obtain almost everywhere convergence for spline
projections of $L^1$ functions:
\begin{theorem}[\cite{ShadrinPassenbrunner2014}]
	\label{thm:aenonperiodic}
	For all $f\in L^1[t_\ell,t_{r+1}]$ 
	there exists a subset $A\subset [t_\ell,t_{r+1}]$ of full Lebesgue measure such that 
	for all sequences $(\Delta_n)$ of partitions of $[t_\ell,t_{r+1}]$ 
	such that $|\Delta_n|\to
	0$, we have
%	Let $f\in L^1[t_\ell,t_r]$ and let $(\Delta_n)$ be a sequence of knot sequences  in $[t_\ell,t_r]$ 
%	as in ?? such
%	that the maximal mesh width $|\Delta_n|$ tends to zero. Then
	\begin{equation*}
		\lim_{n\to\infty}P_{\Delta_n} f (x) = f(x),\qquad x\in A.
	\end{equation*}
\end{theorem}
%\todo{ueberarbeiten}
%Our aim in this article is to prove an analogue of 
%Theorem \ref{thm:aenonperiodic} for orthoprojectors on periodic spline spaces.
%In order to prove a periodic version of this theorem, we also need a periodic version of Theorem
%\ref{thm:shadrin}. This can be proved by first establishing the same assertion for infinite point sequences and
%then by viewing periodic functions as defined on the whole real line [A. Shadrin, private communication].
%The proof of Theorem~\ref{thm:shadrin} for infinite 
%point sequences is announced in
%\cite{Shadrin2001} and carried out \cite{deBoor2012}. In this article, we present a different proof of the
%periodic version of Shadrin's theorem, which directly passes from the interval case to the periodic result 
% without recourse to infinite point sequences.
 %Let us remark that the proof of this theorem relies heavily on estimate \eqref{eq:improvedversion}. 

Our aim in this article is to prove an analogue of 
Theorem \ref{thm:aenonperiodic} for orthoprojectors on periodic spline spaces. In this case, we do not have a
periodic version of \eqref{eq:improvedversion} at our disposal, since the proof of this inequality
does not carry over to the periodic setting.
However, by comparing orthogonal projections onto periodic spline spaces to suitable non-periodic projections, we
are able to obtain a periodic version of Theorem~\ref{thm:aenonperiodic}.

In the course of the proof of the periodic version of Theorem~\ref{thm:aenonperiodic}, 
we also need a periodic version of Theorem
\ref{thm:shadrin}, which can be proved by first establishing the same assertion for infinite point sequences and
then by viewing periodic functions as defined on the whole real line [A. Shadrin, private communication].
The proof of Theorem~\ref{thm:shadrin} for infinite 
point sequences is announced in
\cite{Shadrin2001} and carried out \cite{deBoor2012}. 
In this article we give a different proof of the periodic version of Shadrin's theorem by employing a
similar comparison of periodic and non-periodic projection operators as in the proof of the periodic
version of Theorem \ref{thm:aenonperiodic}.
%
%we also show that the same idea that is used in the proof of the periodic version 
%of Theorem \ref{thm:aenonperiodic} of 
%comparing periodic orthogonal projections to suitable non-periodic ones also leads to a proof of the
%periodic version of Shadrin's theorem. 
This proof directly passes from the interval case to the periodic result 
 without recourse to infinite point sequences.

\subsection{Periodic splines}\label{sec:persplinesdef}
Let $n\geq k$ be a natural number and $\widetilde{\Delta}=(s_j)_{j=0}^{n-1}$ be a sequence of distinct points on
the torus $\mathbb T=\mathbb R / \mathbb Z$ identified
canonically with $[0,1)$, such that for all $j$ we have
\begin{align*}
	s_j &\leq s_{j+1}, \qquad s_j<s_{j+k},
\end{align*}
and we extend $(s_j)_{j=0}^{n-1}$ periodically by 
  \begin{equation*}
   s_{rn+j} = r + s_j 
  \end{equation*}
  for $r\in \mathbb Z\setminus \{0\}$ and $0\leq j\leq n-1$.

%  Let $\tP= \tP_{\widetilde\Delta}$ be the orthogonal projection operator onto square integrable periodic spline
%  functions with grid points $\widetilde\Delta$ and order $k$ on $\mathbb T$. 

Now, the main result of this article reads as follows:

\begin{theorem}\label{thm:aeperiodic} For all functions $f\in L^1(\mathbb T)$ there exists a set $\widetilde{A}$ of
	full Lebesgue measure
	such that for all sequences of partitions
	 $(\widetilde{\Delta}_n)$  
	on $\mathbb T$ as above with $|\widetilde\Delta_n|\to 0$, we have 
%	such that the maximal mesh width $|\widetilde{\Delta}_n|$ tends to zero.
%Then, for every $f\in L^1(\mathbb T)$,
\begin{equation*}
	\lim_{n\to\infty}\tP_{n} f(x) = f(x),\qquad x\in \widetilde A, 
\end{equation*}
where $\tP_n$ denotes the orthogonal projection operator
onto the periodic spline space of order $k$ with knots $\widetilde
\Delta_n$.
\end{theorem}

In order to prove this result, we also need a periodic version of Theorem \ref{thm:shadrin}:

\begin{theorem}\label{th:boundedperiodic}
	There exists a constant $c_k$ depending only on the spline order $k$ such that for all
	knot sequences $\widetilde\Delta=(s_j)_{j=0}^{n-1}$ on $\mathbb T$, the associated orthogonal projection
	operator $\widetilde{P}$ satisfies the inequality
\begin{equation*}
	\|\widetilde P: L^\infty(\mathbb T)\to L^\infty(\mathbb T)\| \leq c_k.
\end{equation*}
\end{theorem}

The idea of the proofs of Theorems \ref{thm:aeperiodic} and \ref{th:boundedperiodic} is to estimate the difference
between the periodic projection operator $\tP$ and the non-periodic projection operator $P$ for certain
non-periodic point sequences associated to $\widetilde\Delta=(s_i)_{i=0}^{n-1}$.

The article is organized as follows. In Section \ref{sec:lemmaBsplines}, we prove a simple lemma on the growth
behaviour of
linear combinations of non-periodic B-spline functions which is frequently needed later in the proofs of both Theorem
\ref{thm:aeperiodic} and Theorem~\ref{th:boundedperiodic}.
Section \ref{sec:shadrinperiodic} is devoted to the proof of Theorem \ref{th:boundedperiodic}, which is needed for
the proof of Theorem \ref{thm:aeperiodic} in Section \ref{sec:aeperiodic}.
Finally, in Section \ref{sec:doubly}, we also apply our method of proof to recover Shadrin's
theorem for infinite point sequences (see \cite{deBoor2012,Shadrin2001}).

\section{A simple upper estimate for B-spline sums} \label{sec:lemmaBsplines}
Let $A$ be a subset of $[t_\ell, t_{r+1}]$. Then, define the set of indices $\indset (A)$ whose B-spline supports
intersect with $A$ as 
\begin{equation*}
	\indset (A) := \{i : A\cap \supp N_i \neq \emptyset\}.
\end{equation*}
We also write $\indset(x)$ for $\indset(\{x\})$. If we have two subsets $U,V$ of indices, we write $d(U,V)$
for the distance between $U$ and $V$ induced by the metric $d(i,j) = |i-j|$.

We will use the notation $A(t)\lesssim B(t)$ to indicate the existence of a constant $C$ that depends only on the
spline order $k$ such that for all $t$ we have $A(t)\leq C B(t),$ where $t$ denotes all explicit or implicit
dependencies that the expressions $A$ and $B$ might have.

The fact that B-spline functions are localized, so a fortiori the set $\indset(x)$ is localized for any
$x\in[t_\ell,t_{r+1}]$,
can be used to derive the following lemma:
\begin{lemma} \label{lem:baseineq}
	Let $J$ be a subset of the index set $\{\ell,\ell+1,\ldots,r-1,r\}$,
	$f= \sum_{j \in J} \langle h,N_j\rangle N_j^*$ and $p\in [1,\infty]$. 
	Then, for all $x\in [t_\ell, t_{r+1}]$, we have the estimate
	\begin{align*}
		|f(x)| &\lesssim \gamma^{d(\indset(x),J)} \|h\|_p \max_{m\in\indset(x),j\in
		J}\frac{\kappa_j^{1/p'}}{h_{j m}}\\
		&\leq\gamma^{d(\indset(x),J)}\|h\|_{p}\max_{m\in \indset(x), j\in J}
		\big(\max\{\kappa_m,\kappa_j\}\big)^{-1/p} \\
		&\leq \gamma^{d(\indset(x),J)}\|h\|_p\cdot |I(x)|^{-1/p},\qquad 1\leq p\leq \infty,
	\end{align*}
	where $\gamma\in (0,1)$ is the constant appearing in \eqref{eq:improvedversion}, $I(x)$ is the interval
		$I=[t_i,t_{i+1})$ containing the point $x$ and the exponent $p'$ is such that $1/p+1/p'=1$.
\end{lemma}
\begin{proof}
	Since $N_j^* = \sum_{m} a_{jm}N_m$,
\begin{equation*}
	f(x) = \sum_{j\in J} \sum_{m\in\indset(x)} a_{j m}\langle h, N_j\rangle N_m(x).
\end{equation*}
This implies
\begin{align*}
	|f(x)| \lesssim \max_{m\in \indset(x)} \Big(\sum_{j\in J} \frac{\gamma ^{|j-m|}}{h_{j m}} 
	\|h\|_{p} \|N_j\|_{p'}\Big),
\end{align*}
where we used inequality \eqref{eq:improvedversion} for $a_{jm}$,
H\"older's inequality with the conjugate exponent $p'=p/(p-1)$ to $p$ and the
fact that the B-spline functions $N_m$ form a partition of unity. Using again the uniform boundedness of $N_j$, we
obtain
\begin{equation*}
	|f(x)| \lesssim \max_{m\in\indset(x)}\Big(\sum_{j\in J} 
	\frac{\gamma^{|j-m|}}{h_{j m}}  \|h\|_p\kappa_j^{1/p'}\Big).
\end{equation*}
Summing the geometric series now yields the first estimate. The second and the third estimate are direct
consequences of the first one.
\end{proof}
\begin{remark}
	We note that we directly obtain the second estimate in the above lemma if we use the weaker inequality 
	\eqref{eq:decayineq} instead of \eqref{eq:improvedversion}.
\end{remark}

\section{The periodic spline orthoprojector is uniformly bounded on $L^\infty$}\label{sec:shadrinperiodic}
In this section, we give a direct proof of Theorem \ref{th:boundedperiodic} on the boundedness of 
periodic spline projectors 
without recourse to infinite knot sequences. Here, we will only use the geometric decay of the matrix
$(a_{jm})$ defined above for splines on an interval.

A vital tool in the proofs of both Theorem \ref{thm:shadrin} and Theorem \ref{thm:aenonperiodic} are B-spline
functions. We will also make extensive use of them and introduce their periodic version, cf.
\cite{Schumaker2007}.
Associated to the periodic point sequence $(s_j)_{j=0}^{n-1}$ and its periodic extension
as in Section~\ref{sec:persplinesdef} we define the non-periodic
  point sequence
  \begin{align*}
	  t_{j} = s_{j},\qquad \text{for }j=-k+1,\ldots,n+k-1
  \end{align*}
  and denote the corresponding non-periodic B-spline functions by $(N_{j})_{j=-k+1}^{n-1}$ with $\supp N_j =
  [t_j,t_{j+k}]$. Then we define for $x\in [0,1)$
  \begin{equation*}
	  \tN_j(x) = N_j(x),\qquad j=0,\ldots,n-k,
  \end{equation*}
  if we canonically identify $\mathbb T$ with $[0,1)$. Moreover, for $j=n-k+1,\ldots,n-1$,
	  \begin{equation*}
		  \tN_j(x) = \begin{cases}
			  N_{j-n}(x),& \text{if }x\in [0,s_j], \\
			  N_j(x), & \text{if }x\in (s_j,1).
		  \end{cases}
	  \end{equation*}

	  We denote by $\tP$ the orthogonal projection operator onto the space of periodic splines of order $k$
	  with knots $(s_j)_{j=0}^{n-1}$, which is the linear span of the B-spline functions
	  $(\tN_j)_{j=0}^{n-1}$ and 
 % The collection of periodic 
 % B-spline functions
 % $(\widetilde N_j)$ forms a basis of the range of $\widetilde P$.
 % \todo{passts noch?}
similarly to the non-periodic case we define 
\begin{equation*}
	\indset(A)=\{0\leq j\leq n-1 : A\cap \supp \widetilde N_j\neq \emptyset\},\qquad A\subset \mathbb T.
\end{equation*}
\begin{lemma}\label{lem:permain}
	Let $f_i$ be a function on $\mathbb T$ with $\supp f_i \subset [s_i,s_{i+1}]$ for some index $i$ in the
	range $0\leq i\leq n-1$. Then,
		for any $x\in \mathbb T$,
		\begin{equation*}
			|\tP f_i(x)| \lesssim \gamma^{\td(\indset(x),\indset(\supp f_i))} \|f_i\|_\infty,
		\end{equation*}
		where $\td$ is the distance function induced by the canonical metric in $\mathbb
		Z/n\mathbb Z$ and $\gamma\in(0,1)$ is the constant appearing in inequality
		\eqref{eq:improvedversion}.
\end{lemma}

\begin{proof} 
	We assume that the index $i$ is chosen such that $s_i<s_{i+1}$, since if $s_i=s_{i+1}$, the function $f_i$ is
	identically zero in $L^\infty$.

	Given a function $f$ on $\mathbb T$, we associate a non-periodic function $Tf$ defined on $[s_i,s_{i+n+1}]$
	given by 
%  \begin{equation*}
%    Tf : [s_{j}, s_{j+n+1}]\mapsto \mathbb R,
%  \end{equation*}
%  which is defined by 
  \begin{equation*}
    Tf(t) = f(\pi(t)), \qquad t\in [s_{i},s_{i+n+1}],
  \end{equation*}
  where $\pi(t)$ is the quotient mapping from $\mathbb R$ to $\mathbb T$.
  We observe that $T$ is a linear operator,
  $\| T : L^2(\mathbb T) \to L^2([s_i,s_{i+n+1}]) \| = \sqrt 2$
  and $\| T : L^\infty(\mathbb T) \to L^\infty([s_i,s_{i+n+1}]) \| =1 $. 
  Moreover, for $x\in \mathbb T$, let $r(x)$ be the representative of $x$ in the interval $[s_i,s_{i+n})$. We want to
	  estimate $\tP f_i(x)$. In order to do this, we first 
	  decompose
	  \begin{equation}
		  \label{eq:decomp1}
		  \tP f_i(x) = T\tP f_i(r(x)) =  PTf_i(r(x))+(T\tP f_i - PTf_i)(r(x)),
	  \end{equation}
	  where $P$ is the orthogonal projection operator onto the space of splines of order $k$ corresponding to
	  the point sequence $\Delta = (t_j)_{j=-k+1}^{n+k}$ associated to the 
	  non-periodic grid points in the
	  interval $[s_i,s_{i+n+1}]$, i.e.,
  \begin{align*}
	  t_j &= s_{i+j},\qquad j=0,\ldots,n+1, \\
    t_{-k+1} &= \cdots = t_{-1} = s_{i},\qquad t_{n+2} =
    \cdots = t_{n+k} =
    s_{i+n+1}.
  \end{align*} 
  Also, let $(N_j)_{i={-k+1}}^{n}$ be the $L^\infty$-normalized B-spline basis corresponding to this point
  sequence.

  We estimate the first term $PTf_i(r(x))$ 
  from the decomposition in \eqref{eq:decomp1} of $\tP f_i(x)$. Since $P$ is
  a projection operator onto splines on an interval, we use representation \eqref{eq:projBsplines} to get 
  \begin{equation*}
	  PTf_i(r(x)) = \sum_{j=-k+1}^{n} \langle Tf_i, N_j\rangle N_j^* (r(x)), 
  \end{equation*}
  and, since $\supp Tf_i\subset [s_i,s_{i+1}]\cup [s_{i+n},s_{i+n+1}] = [t_0,t_1]\cup [t_n,t_{n+1}]$ by definition of $f_i$ and $T$ and $\supp
  N_j \subset [t_j,t_{j+k}]$ for all $j=-k+1,\ldots,n$,
  \begin{equation*}
	  PTf_i(r(x)) = \sum_{j\in J_1} \langle Tf_i, N_j\rangle N_j^* (r(x)),
  \end{equation*}
  with $J_1= \{-k+1,\ldots,0\} \cup \{n-k+1,\ldots,n\}$.
  Employing now Lemma \ref{lem:baseineq} with $p=\infty$ to this sum, we obtain
  \begin{equation}\label{eq:firstpart}
	  |PTf_i(r(x))| \lesssim \gamma^{d(\indset(r(x)),J_1)} \|Tf_i\|_\infty \lesssim 
	  \gamma^{\td(\indset(x),\indset(\supp f_i))} \|f_i\|_\infty.
  \end{equation}

  Now we turn to the second term on the right hand side of \eqref{eq:decomp1}.
  Let $g:= (T\tP - PT) f_i$. Observe that $g\in \mathcal S_k(\Delta)$ since the range of both $T\tP$ and $P$ is
  contained in $\mathcal S_k(\Delta)$. Moreover,
  \begin{equation*}
	  \langle (T\tP - T)f_i, N_j\rangle = \langle \tP f_i - f_i, \widetilde{N}_{j+i}\rangle,\qquad j =
	  0,\ldots,n-k+1,
  \end{equation*}
  where we take the latter subindex $j+i$ to be modulo $n$. This equation is true in the given range of the
  parameter $j$, since in this case, the functions $N_j$ and $\widetilde{N}_{j + i}$ coincide on their
  supports. The fact that $\tP$ is an orthogonal projection onto the span of the functions
  $(\widetilde{N}_j)_{j=0}^{n-1}$ then implies
  \begin{equation*}
	  \langle T\tP f_i - Tf_i, N_j\rangle = \langle\tP f_i - f_i,\tN_{j+i}\rangle=0,\qquad j=0,\ldots,n-k+1.
  \end{equation*}
  Combining this with the fact
  \begin{equation*}
	  \langle PTf_i - Tf_i, N_j\rangle=0,\qquad j=-k+1,\ldots,n,
  \end{equation*}
  since $P$ is an orthogonal projection onto a spline space as well, we obtain that 
  \begin{equation*}
	  \langle g, N_j\rangle =0,\qquad j=0,\ldots n-k+1.
  \end{equation*}
  Therefore, we can expand $g$ as a B-spline sum
  \begin{equation*}
	  g = \sum_{j\in J_2}\langle g,N_j\rangle N_j^*,
  \end{equation*}
  with $J_2=\{-k+1,\ldots,-1\}\cup \{n-k+2,\ldots,n\}$. Now, we employ Lemma \ref{lem:baseineq} on the function $g$
  with the parameter $p=2$ to get for the point $y=r(x)$
  \begin{align*}
	  |g(y)| &\lesssim \gamma^{d(\indset(y),J_2)} \|g\|_2 \max_{j\in J_2} |\supp N_j|^{-1/2}.
  \end{align*}
  Since $g = (T\tP - PT) f_i$ and the operator $T\tP -PT$ has norm $\leq 2\sqrt 2$ on $L^2$, we get
  \begin{equation*}
	  |g(y)| \lesssim \gamma^{d(\indset(y),J_2)} \|f_i\|_2 |\supp f_i|^{-1/2},
  \end{equation*}
  where we also used the fact that $\supp N_j \supset [s_i,s_{i+1}]=[t_0,t_1]$ or $\supp N_j \supset
  [s_{i+n},s_{i+n+1}]=[t_n,t_{n+1}]$ for $j\in J_2$. Since  $d(\indset(y),J_2) \geq
  \widetilde{d}(\indset(x),\indset(\supp f_i))$ and $\|f_i\|_2 \leq \|f_i\|_\infty |\supp f_i|^{1/2}$, we finally get
  \begin{equation*}
	  |g(y)| \lesssim \gamma^{\widetilde{d}(\indset(x),\indset(\supp f_i))} \|f_i\|_\infty.
  \end{equation*}
  Looking at \eqref{eq:decomp1} and combining the latter estimate with \eqref{eq:firstpart}, the proof is
  completed.
\end{proof}
This lemma can be used directly to prove Theorem \ref{th:boundedperiodic} on the uniform boundedness of periodic orthogonal spline projection
operators on $L^\infty$:
\begin{proof}[Proof of Theorem \ref{th:boundedperiodic}]
	We just decompose the function $f$ as $f = \sum_{i=0}^{n-1} f \cdot \mathbb 1_{[s_i,s_{i+1})}$ and apply 
		Lemma~\ref{lem:permain} to each summand and the assertion $\|\tP f\|_\infty \lesssim \|f\|_\infty$
		follows after summation of a geometric series.
\end{proof}

\begin{remark}\label{rem:adjoint}
	(i) Since $\tP$ is a selfadjoint operator, Theorem \ref{th:boundedperiodic}
	also implies that $\tP$ is bounded as an operator
	from $L^1(\mathbb T)$ to $L^1(\mathbb T)$ by the same constant $c_k$ as in the above theorem. Moreover, by
	interpolation, $\tP$ is also bounded by $c_k$ as an operator from $L^p(\mathbb T)$ 
	to $L^p(\mathbb T)$ for any $p\in [1,\infty]$.

	(ii) In the proof of Lemma~\ref{lem:permain}, we only use the second inequality of 
	Lemma~\ref{lem:baseineq} which follows from inequality \eqref{eq:decayineq} 
	on the inverse of the B-spline Gram matrix and does not need its stronger form \eqref{eq:improvedversion}.
	Similarly to the equivalence of Shadrin's theorem and \eqref{eq:decayineq} in the non-periodic case, we can
	derive the equivalence of Theorem \ref{th:boundedperiodic} and the estimate
	\begin{equation*}
		|\widetilde{a}_{ij}| \leq \frac{K\gamma^{\widetilde{d}(i,j)}}{\max(\widetilde{\kappa}_i
		,\widetilde{\kappa}_j)}, \qquad 0\leq i,j\leq n-1,
	\end{equation*}
	where $(\widetilde{a}_{ij})$ denotes the inverse of the Gram matrix $(\langle \tN_i,\tN_j\rangle)$,
	$K>0$ and $\gamma\in(0,1)$ are constants only depending on the spline order $k$,
	$\widetilde{\kappa}_i$ denotes the length of the support of $\tN_i$ and $\widetilde{d}$ is the canonical
	distance in $\bZ/n\bZ$. The proof of this equivalence uses the same tools as the proof in the non-periodic
	case: a periodic version of both Demko's theorem and de Boor's stability.
\end{remark}

\section{Almost everywhere convergence}\label{sec:aeperiodic}
In this section we prove Theorem \ref{thm:aeperiodic} on the a.e. convergence of periodic spline projections.
\begin{proof}[Proof of Theorem \ref{thm:aeperiodic}]
	Without loss of
	generality, we assume that $\widetilde\Delta_n$ has $n$ points. Let
	$\widetilde\Delta_n= (s_j^{(n)})_{j=0}^{n-1}$ 
	 and $(\tN_j^{(n)})_{j=0}^{n-1}$ be the corresponding periodic B-spline functions.
	 Associated to it, define  the non-periodic point 
	 sequence $\Delta_n=(t_j^{(n)})_{j=-m}^{n+k-1}$ with the boundary
	points $0$ and $1$ as
	\begin{align*}
		t_{j}^{(n)} &= s_{j}^{(n)},\qquad j=0,\ldots,n-1, \\
		t_{-m}^{(n)} &= \cdots = t_{-1}^{(n)} = 0, \qquad t_{n}^{(n)} = \cdots = t_{n+k-1}^{(n)} = 1.
	\end{align*}
	We choose the integer $m$ such that the multiplicity of the point $0$ in $\Delta_n$ is $k$
	and denote by $(N_j^{(n)})_{j=-m}^{n-1}$ the non-periodic B-spline functions corresponding to
	this point sequence and by $P_n$ the orthogonal projection operator onto the span of
	$(N_j^{(n)})_{j=-m}^{n-1}$.
	
	We will show that $\tP_n f(x)\to f(x)$ for all $x$ in the set $A$ from Theorem \ref{thm:aenonperiodic} of
	full Lebesgue measure such that $\lim P_n T f(x)=Tf(x)$ for all $x\in A$, where $T$ is just 
	the operator that canonically
	identifies a function defined on $\mathbb T$ with the corresponding function defined on $[0,1)$ and we write
		$x$ for a point in $\mathbb T$ as well as for its representative in the interval $[0,1)$.

	So, choose an arbitrary (non-zero) point 
	$x\in A$ 
	and  decompose $\tP_nf(x)$:
	\begin{equation}
		\label{eq:decomp12}
		\tP_n f(x) = T\tP_nf(x) = P_nTf(x) + \big(T\tP_n f(x) - P_n Tf(x)\big).
	\end{equation}
	%where $T$ is just the operator that canonically identifies a function defined on $\mathbb T$ with the
	%corresponding function defined on $[0,1)$ and we write $x$ for a point in $\mathbb T$ as well as 
	%	for its representative in the
	%	interval $[0,1)$.
	
			For the first term of \eqref{eq:decomp12}, $P_nTf(x)$, we have that $\lim_{n\to\infty}
			P_nTf(x) = Tf(x) = f(x)$ since $x\in A$.

		It remains to estimate the second term $g_n(x) =
		T\tP_nf(x) - P_n T f(x) = T\tP_n f(x) - Tf(x) + Tf(x)- P_n Tf(x)$ of \eqref{eq:decomp12}.
		First, note that $g_n\in
			\mathcal S_k(\Delta_n)$. Moreover,
			\begin{equation*}
				\langle T\tP_n f - Tf, N_j^{(n)}\rangle = \langle \tP_n f-f,\tN_j^{(n)}\rangle
				=0,\qquad j=0,\ldots,n-k-1,
			\end{equation*}
			since $\tP_n$ is the projection operator onto the span of the B-spline functions
			$(\widetilde N_j^{(n)})$, and 
			\begin{equation*}
				\langle Tf - P_nT f, N_j^{(n)}\rangle = 0,\qquad j=-m,\ldots,n-1,
			\end{equation*}
			since $P_n$ is the projection operator onto the span of the functions $(N_j^{(n)})$.
			Therefore, $g_n\in \mathcal S_k(\Delta_n)$ can be written as 
			\begin{equation*}
				g_n = \sum_{j\in J_n} \langle g_n,N_j^{(n)}\rangle N_j^{(n)*},
			\end{equation*}
			with $J_n = \{-m, \ldots, -1\} \cup \{n-k,\ldots,n-1\}$ and $(N_j^{(n)*})$ being the dual
			basis to $(N_j^{(n)})$. We now apply Lemma \ref{lem:baseineq} with $p=1$ to $g_n$ and get
			\begin{equation*}
				|g_n(x)| \lesssim \gamma^{d(\indset_n(x),J_n)}\|g_n\|_1 
				\max_{\ell\in\indset_n(x),j\in
				J_n} \frac{1}{h_{\ell j}^{(n)}},
			\end{equation*}
			where $h_{\ell j}^{(n)}$ denotes the length of the convex hull of $\supp N_\ell^{(n)}\cup
			\supp N_j^{(n)}$ and $\indset_n(x)$ is the set of indices $i$ such that $x$ is in the
			support of $N_i^{(n)}$. Since for $\ell\in \indset_n(x)$, 
			the point $x$ is contained in $\supp
			N_\ell^{(n)}$  and for $j\in J_n$ either the point $0$ or the
			point $1$ is contained in $\supp N_j^{(n)}$, we can further estimate
			\begin{equation*}
				|g_n(x)| \lesssim \gamma^{d(\indset_n(x),J_n)} \|g_n\|_1 \frac{1}{\min(x,1-x)}.
			\end{equation*}
			Now, $\|g_n\|_1 = \|(T\tP_n-P_nT)f\|_1 \lesssim \|f\|_1$, since the operator $T$ has norm
			one on $L^1$ and $\tP_n$ and $P_n$ are both bounded on $L^1$ uniformly in $n$ by Theorem
			\ref{th:boundedperiodic} (cf. Remark \ref{rem:adjoint}) and Theorem \ref{thm:shadrin},
			respectively. Since $|\widetilde \Delta_n|$ tends to zero, and a fortiori 
			the same is true for $|\Delta_n|$, we have that $d(\indset_n(x),J_n)$ tends to infinity 
			as $n\to\infty$. This implies  $\lim_{n\to\infty} g_n(x) = 0$, and therefore, by
			the choice of the point~$x$ and decomposition~\eqref{eq:decomp12}, 
			$\lim \tP_n f(x) = f(x)$. Since $x \in A$ was arbitrary and
			$A$ is a set of full Lebesgue measure, we obtain
			\begin{equation*}
				\lim_{n\to\infty} \tP_n f(y) = 0,\qquad \text{for a.e. }y\in\mathbb T,
			\end{equation*}
			and the proof is completed.
\end{proof}

\section{The case of infinite point sequences}\label{sec:doubly}
In this last section, we use the methods introduced in the previous sections to recover Shadrin's theorem for
infinite point sequences (see \cite{Shadrin2001,deBoor2012}).

Let $(s_i)_{i\in\mathbb Z}$ be a biinfinite point sequence in $\mathbb R$ satisfying  
\begin{align*}
	s_i &\leq s_{i+1}, \qquad s_i<s_{i+k},
\end{align*}
with the corresponding B-spline
functions $(\tN_i)_{i\in\mathbb Z}$ satisfying $\supp \tN_i = [s_i,s_{i+k}]$. 
Furthermore, we denote by $\tP$ the orthogonal projection
operator onto the closed linear span
of the functions $(\tN_i)_{i\in \mathbb Z}$. 
\begin{lemma}\label{lem:biinfinite}
	Let $f$ be a function on $(\inf s_i,\sup s_i)$ with compact support. Then, for any 
	$x\in (\inf s_i, \sup s_i)$,
	\begin{equation*}
		|\tP f(x)| \lesssim \gamma^{d(\indset(x),\indset(\supp f))} \|f\|_\infty,
	\end{equation*}
	where $\gamma\in(0,1)$ is the constant appearing in inequality \eqref{eq:improvedversion}.
\end{lemma}
\begin{proof}
	For notational simplicity, we assume in this proof that the sequence $(s_i)$ is strictly increasing.
	Let $x\in (\inf s_i, \sup s_i)$ and let $I(x)$ be the interval $I = [s_i,s_{i+1})$ containing $x$.
		Since $f$ has compact support and the sequence $(s_i)$ is biinfinite, we can choose the
		indices $\ell$ and $r$ such that $\{x\}\cup\supp f\subset [s_\ell,s_{r+1})$ and 
			with $J=\{\ell-k+1,\ldots,\ell-1\}\cup \{r-k+2,\ldots,r\}$, the inequality
		\begin{equation*}
			\gamma^{d(\indset(x),J)}|\supp f|^{1/2} |I(x)|^{-1/2} \leq \gamma^{d(\indset(x),i(\supp f))}
		\end{equation*}
		is true.

Next, define the point sequence $\Delta=(t_i)_{i=\ell-k+1}^{r+k}$ by 
\begin{align*}
	t_i &= s_i,\qquad i=\ell,\ldots,r+1,\\
	a = t_{\ell-k+1} &= \cdots = t_\ell=s_\ell,\qquad b= t_{r+k} = \cdots = t_{r+1}=s_{r+1},
\end{align*}
and let the collection $(N_i)_{i=\ell-k+1}^r$ be the corresponding B-spline functions and $P$ the
associated orthogonal projector. Let $T$ be the operator that
restricts a function defined on $(\inf s_i,\sup s_i)$ to the interval $[a,b]$. In order to estimate $\tP f(x)$,
we decompose
\begin{equation}\label{eq:decompbiinfinite}
	\tP f(x) = T\tP f(x) = P T f(x) + \big(T\tP f(x) - P T f(x)\big).
\end{equation}
Observe that $P T f = \sum_{n\in F}\langle f,N_n\rangle N_n^*$, where $F=\indset(\supp f)=\{i : \supp f\cap 
\supp N_i\neq \emptyset\}$. Applying Lemma \ref{lem:baseineq} with the exponent $p=\infty$, we obtain
\begin{equation*}
	|PT f(x)| \lesssim  \gamma^{d\left(\indset(x),F\right)}\|f\|_\infty.
\end{equation*}

We now consider the second part of the decomposition \eqref{eq:decompbiinfinite}, the function $g = (T\tP - P T)f = (T\tP - T + T - P T)f$. We observe that $g\in \mathcal S_k(\Delta)$ and, moreover,
\begin{align*}
	\langle T\tP f - Tf , N_j\rangle = \langle \tP f - f, \tN_j\rangle = 0, \qquad j= \ell,\ldots, r-k+1,
\end{align*}
by definition of the projection operator $\tP$, and,
\begin{equation*}
	\langle T f - P T f, N_j\rangle = 0, \qquad j= \ell -k +1,\ldots, r, \end{equation*}
by definition of the projection operator $P$.
Therefore, we can write the function $g$ as 
\begin{equation*}
	g = \sum_{j\in J} \langle g,N_j\rangle N_j^*
\end{equation*}
with $J= \{\ell - k+1 ,\ldots, \ell-1\} \cup \{r-k+2,\ldots, r\}$ as defined above.
Now, by Lemma \ref{lem:baseineq} with the exponent $p=2$, we get 
\begin{align*}
	|g(x)| &\lesssim \gamma^{d(\indset(x),J)}\|g\|_2 \cdot |I(x)|^{-1/2} \lesssim 
	\gamma^{d(\indset(x),J)} \|f\|_2
	\cdot |I(x)|^{-1/2} \\
	&\leq  \gamma^{d(\indset(x),J)}|\supp f|^{1/2}|I(x)|^{-1/2} \|f\|_\infty.
\end{align*}
Finally, due to the choice of $\ell$ and $r$, 
\begin{equation*}
	\gamma^{d(\indset(x),J)}|\supp f|^{1/2} |I(x)|^{-1/2} \leq \gamma^{d(\indset(x),\indset(\supp f))}, 
\end{equation*}
which proves the lemma.
\end{proof}
We can now use this lemma to define $\tP f$ for functions $f\in L^\infty(\inf s_i,\sup s_i)$ that are not
necessarily in $L^2(\inf s_i,\sup s_i)$ if $\inf s_i=-\infty$ or $\sup s_i=+\infty$. If we let $f_i :=
f\charfun_{[s_i,s_{i+1})}$, then $f_i$ has compact support and the above lemma implies that the series
	\begin{equation*}
		\tP f(x) := \sum_{i\in\mathbb Z} \tP f_i(x),\qquad x\in (\inf s_i, \sup s_i),
	\end{equation*}
	is absolutely convergent and, moreover, there exists a constant $C$ only depending on the spline order $k$
	such that 
	\begin{equation*}
		\|\tP f\|_\infty \leq C \|f\|_\infty.
	\end{equation*}
This operator enjoys the characteristic property of an orthogonal projection:
\begin{equation*}
	\langle \widetilde P f - f, \tN_i\rangle =0,\qquad i\in\mathbb Z.
\end{equation*}
\begin{remark}
	One can combine the proofs of Lemma~\ref{lem:biinfinite} and Lemma~\ref{lem:permain} to also obtain the
	uniform boundedness of the spline orthoprojector on $L^\infty$ for one-sided infinite point sequences.
\end{remark}
\bibliography{persplines}
\bibliographystyle{plain}
\end{document}